\documentclass{amsart}
\usepackage[utf8]{inputenc}
\usepackage[all]{xy}
\usepackage{amsmath}
\usepackage{amssymb,amsthm,multicol,amscd,graphics,graphicx}
\usepackage[foot]{amsaddr}
\usepackage{hyperref}

\newtheorem{ex}{Example}
\newtheorem{thm}{Theorem}
\newtheorem{prop}{Proposition}
\newtheorem{cor}{Corollary}
\newtheorem{lem}{Lemma}
\newtheorem{defn}{Definition}
\newtheorem{rmk}{Remark}

\def\supp {\mathrm{supp}}
\title{On polynomial automorphisms of Nagata type}
\author{Huarcaya, Jorge A. C.  $^{1}$}
\author{Palacios, Joe $^{2,3}$}
\address{$^1$ Universidad Nacional Mayor de San Marcos, UNMSM}
\address{$^2$ Universidad Nacional de Ingeniería, UNI}
\address{$^3$ Instituto de Matemática y Ciencias Afines, IMCA}
\email{jcoripacoh@unmsm.edu.pe, jpalacios@imca.edu.pe}
\date{}

\begin{document}

	\begin{abstract}
			We define a family of polynomial ring homomorphisms generalizing the well-known Nagata automorphism. We establish necessary and sufficient conditions under which these homomorphisms are automorphisms, and verify that they satisfy the Jacobian conjecture.  
		 Additionally, we provide a necessary condition within this family to obtain wild automorphisms, and independently derive a property related to the upper semicontinuity of the \L{}ojasiewicz exponent at infinity. 
	\end{abstract}

	
	
	\maketitle

	\tableofcontents
	
\section{Introduction}

In what follows $F$ will be a field of characteristic $0$ unless otherwise stated. The polynomials 
\begin{equation}\label{nagaut}
	\begin{split}
		f(x,y,z)&=x-2y(zx+y^2)-z(zx+y^2)^2,\\
		g(x,y,z)&=y+z(zx+y^2),\\
		h(x,y,z)&=z.   
	\end{split}
\end{equation}
define an automorphism of $F[x,y,z]$ called Nagata automorphism \cite{Nag72}.
The determinant of the Jacobian matrix of this automorphism turns out to be $1$.

An automorphism is {\em tame} if it is  a finite composition of elementary and linear automorphisms. An automorphism that is not tame is called {\em wild} \cite{SU04}.

In 1972, Nagata conjectured that the aforementioned automorphism is wild. This conjecture was proved in \cite[Corollary 9]{SU04}.

In this work, we define a family of ring homomorphisms on $F[x,y,z]$, given in \eqref{eqnag}, generalizing the Nagata automorphism, which includes the family of automorphisms constructed by Bass \cite{Bass84} in some sense, but is quite different to the family of automorphisms constructed in \cite[(1)]{Kur14}. Our family provides wild automorphisms under some assumptions, see Corollary \ref{corwild}. In Theorem \ref{Teorpr}, we solve a linear partial differential equation, whose solution is related to the non-vanishing of the Jacobian determinant of those automorphisms. As a consequence, we obtain a necessary and sufficient condition on homomorphisms of Nagata type  with nonzero Jacobian determinant to be automorphisms, which is a particular case of the famous Jacobian conjecture, see  Corollary \ref{mainth}.

\bigskip

\noindent{\sc Acknowledgments.} We are grateful to Carles Bivià-Ausina for valuable comments on an early version of this article. 

\section{Nagata polynomial map}
In the following paragraph we recall the notion of tame automorphisms of polynomials.  

Fix a natural number $n\geq 2$. For each index $1\leq i\leq n$, the $i$th {\em elementary automorphism} of $F^n$ on itself is the automorphism defined as 
$$(x_1,\dots,x_n)\mapsto\Big(x_1,\dots, x_{i-1}, x_i+f(x_1,\dots,x_{i-1},x_{i+1},\dots,x_n), x_{i+1},\dots,x_n\Big)\,,$$
where $f$ is a polynomial map in $n-1$ variables, and its inverse is the automorphism 
$$(x_1,\dots,x_n)\mapsto\Big(x_1,\dots, x_{i-1}, x_i-f(x_1,\dots,x_{i-1},x_{i+1},\dots,x_n), x_{i+1},\dots,x_n\Big)\,.$$
It turns out that the Jacobian matrix of each elementary automorphism has determinant $1$. 

Since $F$ is an infinite field, there is a one-to-one correspondence between polynomial maps $F^n\to F^n$ and  $n$-tuples of polynomials in $F[x_1,\dots,x_n]$.   
In what follows, we shall focus on the case $n=3$.

\begin{defn}
	{\em 
		For a polynomial $\varphi\in F[x,y,z]$, we set
		\begin{equation}\label{eqnag}
			\begin{split}
				f(x,y,z)&=x-2y \varphi-z\varphi^2\,,\\
				g(x,y,z)&=y+z\varphi\,,\\
				h(x,y,z)&=z\,.     
			\end{split}
		\end{equation}
		The triplet $(f,g,h)$ induces an ring homomorphism on $F[x,y,z]$ which will be called {\em Nagata homomorphism}.	
	}
\end{defn}
Note that  $(f,g,h)$ is the identity when one takes $\varphi=0$, so $(f,g,h)$ may be seen as a perturbation of the identity. If $\varphi=xz+y^2$ then $(f,g,h)$ is the Nagata automorphism \eqref{nagaut}. If $\varphi=\lambda (xz+y^2)$ with $\lambda\in F$, then we obtain a family of automorphisms similar to those constructed by Bass \cite{Bass84}, \cite[3.8.1]{Freu17}.

The Milnor number of $(f,g,h)$ is the dimension
$$\mu(f,g,h)=\dim_F \frac{F[x,y,z]}{\langle f,g,h\rangle}$$
as vector space over $F$. 
\begin{lem}\label{lemmilnor}
	If $(f,g,h)$ is a Nagata homomorphism, then $\mu(f,g,h)=1$.    
\end{lem}
\begin{proof}
	One has,
	\begin{align*}
		y&=(y+z \varphi)-     \varphi z=g-\varphi z,\\
		x&=(x-2y \varphi-z\varphi^2)+2(y+z \varphi)\varphi-\varphi^2 z=f+2\varphi g-\varphi^2 z,
	\end{align*}
	hence we get an equality $\langle f,g,h\rangle=\langle x,y,z\rangle$ of polynomial ideals, and this implies that $\mu(f,g,h)=1$.   
\end{proof}

\begin{rmk}
	{\em 
	Note that the Nagata homomorphism $(f,g,h)$ is not always an automorphism for any $\varphi$, despite the fact that $\mu(f,g,h)=1$ by Lemma \ref{lemmilnor}. For instance, if $\varphi=x$, then $(f,g,h)$ is not an automorphism, because as a map $(f,g,h)$ sends both points $(0,0,1)$ and $(-1,1,1)$ to $(0,0,1)$. Similarly, $(f,g,h)$ is not an automorphism for $\varphi=y$. In contrast, Corollary \ref{mainth} provides a necessary and sufficient condition on $\varphi$, in order that its corresponding Nagata homomorphism $(f,g,h)$ is an automorphism.
}	
\end{rmk}

\begin{prop}\label{detedp}
	The Jacobian matrix  $J(f,g,h)$ of the Nagata homomorphism $(f,g,h)$, as in \eqref{eqnag}, has nonzero determinant in $F$ if and only if $\varphi$ satisfies the linear partial differential equation:
	$$-2y\varphi_x+z\varphi_y=0\,.$$
\end{prop}
\begin{proof}
	We have
	\begin{align*}
		f_x&=1-2y\varphi_x-2z\varphi\varphi_x,\\
		f_y&=-2\varphi-2y\varphi_y-2z\varphi\varphi_y,\\
		g_x&=z\varphi_x,\\
		g_y&=1+z\varphi_y\,.
	\end{align*}
	Hence the Jacobian matrix $J(f,g,h)$ is 
	\[\begin{pmatrix}
		1-2y\varphi_x-2z\varphi\varphi_x&-2\varphi-2y\varphi_y-2z\varphi\varphi_y&f_z\\
		z\varphi_x&1+z\varphi_y&g_z\\
		0&0&1
	\end{pmatrix}
	\]
	whose determinant is 
	\[\begin{vmatrix}
		1-2y\varphi_x-2z\varphi\varphi_x&-2\varphi-2y\varphi_y-2z\varphi\varphi_y\\
		z\varphi_x&1+z\varphi_y
	\end{vmatrix}=1-2y\varphi_x+z\varphi_y\,.\]
	If $-2y\varphi_x+z\varphi_y=0$ then the determinant of $J(f,g,h)$ is $1$. Conversely, if the determinant of $J(f,g,h)$ is a constant $c\in F-\{0\}$, then 
	$$1-2y\varphi_x+z\varphi_y=c\,.$$
	However this equality is only possible when $c=1$, therefore $-2y\varphi_x+z\varphi_y=0$.
\end{proof}

\begin{prop}\label{automp}
	Let  $p\in F[t_1,t_2]$ and set $\varphi=p(xz+y^2,z)$ in $F[x,y,z]$. Then 
	the corresponding Nagata homomorphism $(f,g,h)$ 
	is an automorphism of $F[x,y,z]$ with inverse:
	\begin{equation}\label{pinverse}
		\begin{split}
			f'(x,y,z)&=x+2y \varphi-z\varphi^{2},\\
			g'(x,y,z)&=y-z\varphi,\\
			h'(x,y,z)&=z\,.
		\end{split}
	\end{equation}
\end{prop}
\begin{proof}
	We have
	\begin{align*}
		\varphi(f',g',h')&=p(f'z+g'^2,z)\\
		&=p\Big((x+2y \varphi-z\varphi^{2})z+(y-z\varphi)^2,z\Big)\\
		&=p(xz+y^2,z)\\
		&=\varphi\,,
	\end{align*}
	thus  $ \varphi(f',g',h')=\varphi$. Hence, 
	\begin{align*}
		f(f',g',h')&=f'-2g'\varphi(f',g',h')-z\varphi^{2}(f',g',h')\\
		&=(x+2y \varphi-z\varphi^{2})-2(y-z\varphi)\varphi-z\varphi^{2}\\
		&=x\,.
	\end{align*}
	On the other hand, we have
	\begin{align*}
		g(f',g',h')&=g'+z\varphi(f',g',h')\\
		&=(y-z\varphi)+z\varphi\\
		&=y\,.
	\end{align*}
	Therefore, the composite $(f,g,h)\circ(f',g',h')$ is the identity. Similarly, one verifies that  $(f',g',h')\circ(f,g,h)$ is also the identity.   
\end{proof}

We give the following definition in order to determine the highest weighted homogeneous component of the polynomial $\varphi=p(xz+y^2,z)$.

\begin{defn}
	{\em Let us fix algebraically independent indeterminates $t_1, \dots, t_n$ over $F$. For $k=(t_1,\dots, t_n)\in \mathbb{Z}^n_{\geq 1}$, we write $t^k$ to denote the monomial $t_1^{k_1}\cdots t_n^{k_n}$. Any polynomial $h\in F[t_1,\dots, t_n]$ can be written as $h=\sum_{k}a_kt^k$. 
		\begin{enumerate}
			\item The
			{\em support} of $h$, denoted by $\supp(h)$, is defined as the set
			\begin{equation}\nonumber
				\supp(h)=\{k\in \mathbb{Z}^n_{\geq 0}: a_k\neq 0\}\,.
			\end{equation}
			\item Let $v=(v_1,\dots,v_n)\in \mathbb{Z}^n_{\geq 1}$. The {\em degree} of $h$ with respect to $v$, denoted by $d_v(h)$, is the maximum of the scalar product $\langle v,k\rangle$ such that $a_{k}\neq 0$. We denote by $h^{
				v}$ the sum of those terms $a_{k}t^{k}$ such that  $\langle v,k\rangle=d_v(h)$. In particular,  for $v=(1,\dots,1)$, we write $$\overline{h}=h^v\,,$$ that is the highest degree homogeneous component of $h$, and in this case we have $d_v(h)=\deg(h)$.
			\item  We say that $h$ is {\em weighted homogeneous} with respect to $v$ when $h^{v}=h$.
		\end{enumerate}
		
	}
\end{defn}

\begin{prop}\label{autwild}
	Let $\varphi\in F[x,y,z]$ and suppose that its corresponding Nagata homomorphism $(f,g,h)$ is an automorphism of $F[x,y,z]$. If one the partial derivatives  $\overline{\varphi}_x, \overline{\varphi}_y$ is nonzero, then the automorphism $(f,g,h)$ is wild.   
\end{prop}
\begin{proof}
	Let $\overline{f}, \overline{g}, \overline{h}$ be the highest homogeneous component of $f,g,h$, respectively.  Since $\overline{\varphi}_x$ or $ \overline{\varphi}_y$ is nonzero, we have $\deg(\overline{\varphi})>0$, so
	\begin{equation*}
		\begin{split}
			\overline{f}&=-z\overline{\varphi}^2\,,\\
			\overline{g}&=z\overline{\varphi}\,,\\
			\overline{h}&=z\,.     
		\end{split}
	\end{equation*}
	Hence,
	\begin{align*}
		\overline{f}_x\overline{g}_y-  \overline{f}_y\overline{g}_x&=0\\
		\overline{f}_y\overline{g}_z-  \overline{f}_z\overline{g}_y&=-z\,\overline{\varphi}^ 2\overline{\varphi}_y\\
		\overline{f}_x\overline{g}_z-  \overline{f}_z\overline{g}_x&=-z\,\overline{\varphi}^ 2\overline{\varphi}_x\\
		&\\
		\overline{g}_x\overline{h}_y-  \overline{g}_y\overline{h}_x&=0\\
		\overline{g}_y\overline{h}_z-  \overline{g}_z\overline{h}_y&=z\,\overline{\varphi}_y\\
		\overline{g}_x\overline{h}_z-  \overline{g}_z\overline{h}_x&=z\,\overline{\varphi}_x\\
		&\\
		\overline{f}_x\overline{h}_y-  \overline{f}_y\overline{h}_x&=0\\
		\overline{f}_y\overline{h}_z-  \overline{f}_z\overline{h}_y&=-2z\,\overline{\varphi}_y\\
		\overline{f}_x\overline{h}_z-  \overline{f}_z\overline{h}_x&=-2z\,\overline{\varphi}_x\,.
	\end{align*}
	By \cite[Corollary 2]{SU05} or \cite[Lemma 1]{SU04}, we deduce that $\overline{f}, \overline{g}, \overline{h}$ are mutually algebraically independent. On the other hand, from the hypothesis $\overline{\varphi}\in F[x,y,z]$ is a homogeneous polynomial of positive degree such that $\overline{\varphi}\not\in F[z]$, it is not difficult to see that none of the polynomials   $\overline{f}, \overline{g}, \overline{h}$ belongs to the subalgebra generated by the other two polynomials. Therefore, by \cite[Corollary 8]{SU04} $(f,g,h)$ is wild. 
\end{proof}

\begin{lem}\label{lemwild} 
	Let $p\in F[t_1,t_2]$ and $v=(2,1)$. If $\varphi=p(xz+y^2,z)\in F[x,y,z]$, then $\deg(\varphi)=d_v(p)$ and $\overline{\varphi}=p^v(xz+y^2,z)$. In particular, if $p$ is weighted homogeneous with respect to $v=(2,1)$, then $\varphi$ is homogeneous.
\end{lem}
\begin{proof}
	Let us write the polymomial $p$ as $p=\sum_k a_{k} t^{k}$. One has
	\begin{align*} 
		\varphi=\sum_k a_{k} (xz+y^2)^{k_1}z^{k_2}=\sum_k a_{k}\sum_{n=0}^{k_1}\binom{k_1}{n}x^{k_1-n}y^{2n}z^{k_1+k_2-n}\,.
	\end{align*}
	From the last equality, we observe that \begin{equation}\label{lem11}
		\supp(\varphi)=\Big\{(k_1-n,2n,k_1+k_2-n): \quad k\in \supp(p)\,\textnormal{and}\,\, 0\leq n\leq k_1\Big\}\,.
	\end{equation}
	Therefore,
	\begin{align} 
		\deg(\varphi)&=\max\Big\{k_1-n+2n+k_1+k_2-n:\quad (k_1,k_2)\in \supp(p)\,\textnormal{and}\, 0\leq n\leq k_1 \Big\}\nonumber\\
		&=\max\Big\{2k_1+k_2:\quad (k_1,k_2)\in \supp(p)\Big\}\nonumber\\
		&=d_v(p)\label{lem12}\,.
	\end{align}
	Hence, by the equalities  given in  (\ref{lem11}) and (\ref{lem12}), the highest degree homogeneous component of $\varphi$ is computed as follows
	\begin{align*} 
		\overline{\varphi}=&\sum_{\substack{k\in \supp(p)\\0\leq n\leq k_1\\2k_1+k_2=d_v(p)}}a_{k}\sum_{n=0}^{k_1}\binom{k_1}{n}x^{k_1-n}y^{2n}z^{k_1+k_2-n}\\
		=&\sum_{\substack{k\in \supp(p)\\0\leq n\leq k_1\\2k_1+k_2=d_v(p)}}a_{k}\left(\sum_{n=0}^{k_1}\binom{k_1}{n}x^{k_1-n}z^{k_1-n}y^{2n}\right)z^{k_2}\\
		=&\sum_{\substack{k\in \supp(p)\\0\leq n\leq k_1\\2k_1+k_2=d_v(p)}}a_{k}\left(\sum_{n=0}^{k_1}\binom{k_1}{n}(xz)^{k_1-n}y^{2n}\right)z^{k_2}\\
		=&\sum_{\substack{k\in \supp(p)\\2k_1+k_2=d_v(p)}}\left(xz+y^2\right)^{k_1}z^{k_2}\\
		=& p^{v}(xz+y^2,z)\,,
	\end{align*}
	and it proves the assertion. 
\end{proof}

\begin{ex}\label{exww}
	{\em 
		If $\varphi(x,y,z)=p(xz+y^2,z)$, then it is not always true that $\overline{\varphi}$ coincides with $ \overline{p}(xz+y^2,z)$. Indeed, consider $p(t_1,t_2)=t_1^2-t_2^3+t_1t_2^2\in F[t_1,t_2]$, hence $\overline{p}=-t_2^3+t_1t_2^2$, and by an immediate computation one has $$\overline{\varphi}=\left(xz+y^2\right)^2+\left(xz+y^2\right)z^2\neq -z^3+\left(xz+y^2\right)z^2=\overline{p}(xz+y^2,z)\,.$$
		On the other hand, it is easy to verify that if $v=(2,1)$, then $d_v(p)=4$ and $p^v=t_1^2+t_1t_2^2$. Finally, by a direct computation one has
		$$
		\overline{\varphi}=p^v(xz+y^2,z)\,.
		$$.  
	}
\end{ex}

\begin{cor}\label{corwild}
	Let  $p\in F[t_1,t_2]$ and set $\varphi=p(xz+y^2,z)$ in $F[x,y,z]$. 
	If the partial derivative $(p^{v})_{t_1}$ is nonzero, where $v=(2,1)$, then its corresponding Nagata homomorphism $(f,g,h)$ is a wild automorphism. 
	
\end{cor}
\begin{proof}
	By Proposition \ref{automp}, $(f,g,h)$ is an automorphism 
	and Lemma \ref{lemwild} establishes the equality $\overline{\varphi}=p^{v}(xz+y^2,z)$. Hence the fact that $(p^v)_{t_1}\neq 0$ implies that $\overline{\varphi}_x\neq 0$ and $\overline{\varphi}_y\neq 0$. By Proposition  \ref{autwild}, the automorphism $(f,g,h)$ is wild.  
\end{proof}

\begin{ex}
	{\em 
		Let $\varphi=(xz+y^2)^n$ with $n\geq 1$. Corollary \ref{corwild} implies that 
		\begin{align*}
			f(x,y,z)&=x-2y (xz+y^2)^n-z(xz+y^2)^{2n},\\
			g(x,y,z)&=y+z(xz+y^2)^n,\\
			h(x,y,z)&=z,
		\end{align*}
		is a wild automorphism. 
	}
\end{ex}
\begin{ex}
	{\em 
		Let $\varphi(x,y,z)=p(xz+y^2,z)$, where $p(t_1,t_2)=t_1^2-t_2^3+t_1t_2^2\in F[t_1,t_2]$ as in Example \ref{exww}. Since $p^v=t_1^2+t_1t_2^2$ and $(p^v)_{t_1}=2t_1+t_1^2\neq0$, Corollary \ref{corwild} implies that corresponding Nagata homomorphism $(f,g,h)$ is a wild automorphism.
	}
\end{ex}

\noindent {\it \L{}ojasiewicz exponent.---}
Let $K$ be a normed field. For $z=(z_1,z_2,\dots,z_m)\in K^m$, we write $|z|=\max_{1\leq i\leq m}|z_i|$. If $F=(F_1,F_2,\dots,F_n)\colon K^m\to K^n$  is a polynomial map, then we set $\deg F=\max_{1\leq i\leq n}\deg F_i$. 
The {\em \L{}ojasiewicz exponent at infinity} of a polynomial map $F\colon K^m\to K^n$, denoted by $\mathcal{L}_{\infty}(F)$, is defined as the supremum of the real numbers $\alpha$ such that there exist two positive real numbers $C$ and $R$ so that 
$$|F(z)|\geq C|z|^{\alpha}$$
for all $|z|\geq R$ in $K^m$. One can follow the argument provided in \cite[(1.1)]{Plos85} to deduce that  
\begin{align}\label{Loj}
	\mathcal{L}_{\infty}(F)=\dfrac{1}{\deg (F^{-1})}
\end{align}
for every polynomial automorphism $F$.

For   $K=\mathbb{R}$ or $K=\mathbb{C}$, it is well known that  $\mathcal{L}_{\infty}(F)$ exists if and only if $F^{-1}(0)$ is compact, and in this case, $\mathcal{L}_{\infty}(F)$ is a rational number  \cite[Proposition 2.6]{Kra07}. The computation of this number is generally considered a highly nontrivial problem \cite{CKT, Len, Plos85,BiviaMZ, BH}.   

A {\em polynomial deformation} of a polynomial $p(t_1,\dots, t_m)\in K[t_1,\dots, t_m]$ is a polynomial   $H(t_1,\dots, t_m, s_1,\dots, s_n)\in K[t_1,\dots, t_m, s_1,\dots, s_n]$ such that $$H(t_1,\dots, t_m, 0,\dots, 0)=p(t_1,\dots, t_m)\,.$$
For simplicity, we write $s=(s_1,\dots, s_n)$, so that the polynomial ring $K[t_1,\dots, t_m, s_1,\dots, s_n]$ shall be shortly written as $K[t_1,\dots, t_m, s]$, and this allows to set $$p_s(t_1,\dots, t_m)=H(t_1,\dots, t_m, s_1,\dots, s_n)\,.$$    

The following is motivated by the lower semicontinuous property of the \L{}ojasiewicz exponent proven for finite holomorphic germs with constant Milnor number \cite[Theorem 2.1]{Plos10}. However, for polynomial maps, we expect the upper semicontinuous property under certain conditions, in view of the fact that Nagata homomorphisms have a constant Milnor number, see Lemma \ref{lemmilnor}.
\begin{cor}
	Let $p_s(t_1,t_2)\in K[t_1,t_2, s]$ be a polynomial deformation of a polynomial $p(t_1,t_2)\in K[t_1,t_2]$. Let $\varphi=p(xz+y^2,z)$, $\varphi_s=p_s(xz+y^2,z)$, and suppose that $(f,g,h)$ and $(f_s,g_s,h_s)$ are their corresponding Nagata automorphism.
	Then 
	$$\mathcal{L}_{\infty}(f_s,g_s, h_s)\leq \mathcal{L}_{\infty}(f,g, h)$$
\end{cor}
\begin{proof}
	Since $\supp(p)\subseteq \supp(p_s)$ with respect to the variables $t_1$ and $t_2$, we have $\deg(p)\leq \deg(p_s)$, and for $v=(2,1)$ one has 
	$$\deg \varphi=d_v(p)\leq d_v(p_s)=\deg \varphi_s\,.$$
	Hence, in view of \eqref{Loj} we have 
	\begin{align*}
		\mathcal{L}_{\infty}(f_s,g_s, h_s)&=\dfrac{1}{\deg ((f_s,g_s, h_s)^{-1})}\\&
		=\dfrac{1}{2\deg \varphi_s+1 }\\&\leq 
		\dfrac{1}{2\deg \varphi+1 }\\&=\dfrac{1}{\deg ((f,g, h)^{-1})}\\&=\mathcal{L}_{\infty}(f,g, h)\,.
	\end{align*}
\end{proof}

\section{Statement of the main results }
The main results in this work are the following:
\begin{thm} \label{Teorpr}
	The general solution in $F[x,y,z]$ of the linear partial differential equation: 
	\begin{equation}\label{eq1}
		-2y\varphi_x+z\varphi_y=0
	\end{equation}
	has the form $
	\varphi=p(xz+y^2,z)$,
	where $p\in F[t_1,t_2]$.
\end{thm}

\begin{cor}\label{mainth}
	Let  $\varphi\in F[x,y,z]$ and let $(f, g, h)$ be its corresponding Nagata homomorphism as in \eqref{eqnag}. 
	Then, the determinant of the Jacobian matrix $J(f,g,h)$ is a nonzero constant in $F$ if and only if $(f,g,h)$ is an automorphism of $F[x,y,z]$.
	In this case, there exists a polynomial $p\in F[t_1,t_2]$ such that $\varphi=p(xz+y^2,z)$, moreover the inverse of $(f,g,h)$ is given in \eqref{pinverse}.
\end{cor}

The proofs of these results are given in Section \ref{sectproofs}.  First of all, we shall provide some preliminary results in the next paragraphs.

\begin{lem} \label{lemahom}
	Let $\varphi=\varphi^0+\varphi^1+\cdots+\varphi^n$ be a solution of the linear partial differential equation: 
	\begin{equation}\label{eq111}
		A(x,y,z)\varphi_x+B(x,y,z)\varphi_y+C(x,y,z)\varphi_z=0\,,
	\end{equation}
	where each $\varphi^{k}\in F[x,y, z]$ is the homogeneous component of $\varphi$ of degree $k$ for $0\leq k\leq n$, and $A,B,C$ are homogeneous polynomials of the same degree. Then each homogeneous component $\varphi^k$ is also a solution of the equation \eqref{eq111}.
\end{lem}

\begin{proof}
	Replacing $\varphi=\sum_{k=0}^{n}\varphi^k$ into the linear PDE (\ref{eq1}), we obtain:
	\begin{equation*}
		\sum_{k=0}^n\Big(A(x,y,z)\varphi^k_x+B(x,y,z)\varphi^k_y+C(x,y,x)\varphi^k_z\Big)=0\,,
	\end{equation*}
	which is a sum of homogeneous polynomials of different degrees. Therefore, 
	\begin{equation*}
		A(x,y,z)\varphi^k_x+B(x,y,z)\varphi^k_y+C(x,y,x)\varphi^k_z=0 \,\,\, \textnormal{for all}\,\, k=0,1,\dots, n\,, 
	\end{equation*}
	as wanted. 
\end{proof}	

By the previous lemma, to find the general solution of the equation (\ref{eq1}), it is enough to deal with its homogeneous polynomials components. 

Let  $\varphi\in F[x,y,z]$ be a homogeneous polynomial of degree $n$, if we consider the canonical isomorphism of rings $F[x,y,z]\cong F[x,y][z]$, then we can suppose  that $\varphi$ has the form: 
\begin{equation}\label{dpch}
	\varphi=\phi^{n}+\phi^{n-1}z+\phi^{n-2}z^2+\dots +\phi^{1}z^{n-1}+\phi^0z^{n}=\sum_{k=0}^{n-1}\phi^{n-k}z^k+\phi^0z^{n}\,,
\end{equation}
with $\phi^{k}\in F[x,y]$ homogeneous polynomials of degree $k$ for $k=0,1,\dots,n$.

\begin{lem}\label{psedp}
	Let $\varphi$ be a homogeneous polynomial in $F[x,y,z]$ of degree $n$ that is a solution of the linear PDE \eqref{eq1}. Then, $\phi^{n}=ay^{n}$ and $\phi^{1}=bx$, where $a$ and $b$ are constants in $F$, and moreover
	$$2y\phi^{n-k-1}_x=\phi^{n-k}_y$$
	for all $k=0,1,\dots, n-2$.
\end{lem}
\begin{proof}
	Let us derivate $\varphi$ with respect to $x$ and  $y$:
	\begin{align*}
		\varphi_x&=\sum_{k=0}^{n-1}\phi^{n-k}_xz^k=\phi_x^{n}+\sum_{k=1}^{n-1}\phi^{n-k}_xz^k\,,\\
		\varphi_y&=\sum_{k=0}^{n-1}\phi^{n-k}_yz^k=\sum_{k=0}^{n-2}\phi^{n-k}_yz^k+\phi^1_y z^{n-1}\,.
	\end{align*}
	Substituting these equalities in the equation (\ref{eq1}), we have
	$$
	-2y\phi_x^{n}+\sum_{k=1}^{n-1}-2y\phi^{n-k}_xz^k+\sum_{k=0}^{n-2}\phi^{n-k}_yz^{k+1}+\phi^1_y z^{n}=0\,,$$
	and by index rescaling, we get
	$$
	-2y\phi_x^{n}+\sum_{k=0}^{n-2}-2y\phi^{n-k-1}_xz^{k+1}+\sum_{k=0}^{n-2}\phi^{n-k}_yz^{k+1}+\phi^1_y z^{n}=0\,.$$
	Hence,
	\begin{align*}
		-2y\phi_x^{n}+\sum_{k=0}^{n-2}\left(-2y\phi^{n-k-1}_x+\phi^{n-k}_y\right)z^{k+1}+\phi^1_y z^{n}=0\,.\label{eqp5}
	\end{align*}
	This equation gives the equalities $\phi^{n}_x=0$ and $\phi^{1}_y=0$, so $\phi^{n}=ay^{n}$ and $\phi^{1}=bx$ for some constants $a,b\in F$.
	Moreover,
	\begin{equation}
		2y\phi^{n-k-1}_x=\phi^{n-k}_y\,\,\, \textnormal{para todo}\,\,\, k=0,1,\dots, n-2\,.\label{eq5}
	\end{equation}
\end{proof}

The following result, will enable us to know  how a solution can be written of the equation (\ref{eq1}).
\begin{prop}\label{pr1}
	Let $\varphi\in F[x,y,z]$ be a homogeneous solution of the linear PDE \eqref{eq1}. 
	We have the following assertions:
	\begin{enumerate}
		\item[ ($a$)] If $\varphi$ has degree $2n$, then it has the form: 
		\begin{equation*}
			\varphi=a(xz+y^2)^n+z\psi\,,
		\end{equation*}
		where $a\in F$ and $\psi\in F[x,y,z]$ is a homogeneous polynomial which is also a solution of the equation \eqref{eq1} of degree $2n-1$.
		\item[($b$)] If $\varphi$ has degree $2n+1$,  then it has the form: 
		\begin{equation*}
			\varphi=z\psi\,,
		\end{equation*}
		where $\psi\in F[x,y,z]$ is homogeneous of degree $2n$ which is also a solution of \eqref{eq1} of degree $2n$.
	\end{enumerate}
	
\end{prop}

\begin{proof}
	($a$). Let us suppose that:
	\begin{equation}\label{eqh}
		\varphi=\phi^{2n}+\phi^{2n-1}z+\cdots+\phi^{1}z^{2n-1}+\phi^0z^{2n}\,,
	\end{equation}
	where $\phi^{2n-k}\in F[x,y]$ is a homogeneous polynomial of degree $2n-k$ for  $0\leq k\leq 2n$. By Lemma \ref{psedp}, we get the following equalities: 
	\begin{align}	
		\phi^{2n}&=ay^{2n}\label{eq7}\,\,\,\, \textnormal{and}\,\,\,\, \phi^{1}=bx\\
		2y\phi^{2n-k-1}_x&=\phi^{2n-k}_y\,\,\, \textnormal{para todo}\,\,\, k=0,1,\dots, 2n-2\,.\label{eq8}
	\end{align}
	We claim that there are constants $a, a^{2n-k}_{k-1-j}$, with  $1\leq k\leq n$ y $0\leq j\leq k-1$, such that:
	\begin{align}
		\phi^{2n-k}&=a\binom{n}{k}y^{2n-2k}x^k+\sum_{j=0}^{k-1} a^{2n-k}_{k-1-j}y^{2n-2k+1+j}x^{k-1-j}\label{eq12}
	\end{align}
	Indeed, taking $k=0$, from the equalities (\ref{eq7}) and  (\ref{eq8}), we find that:
	\begin{equation*}
		2y\phi^{2n-1}_x=\phi^{2n}_y=2nay^{2n-1}\,,
	\end{equation*}
	so that $\phi_x^{2n-1}=nay^{2n-2}$. Hence, integrating with respect to $x$, we obtain
	$$\phi^{2n-1}=nay^{2n-2}x+a^{2n-1}_0y^{2n-1}\,,$$ which can be written as 
	\begin{align*}
		\phi^{2n-1}=a\binom{n}{1}y^{2n-2}x+a^{2n-1}_{0}y^{2n-1}\,,
	\end{align*}
	thus \eqref{eq12} is true for $k=1$. Now, suppose that we have \eqref{eq12}. 
	Taking the derivative with respect to $y$, we have
	$$\phi^{2n-k}_y=2a\binom{n}{k}(n-k)y^{2n-2k-1}x^k+\sum_{j=0}^{k-1}\, (2n-2k+1+j)a^{2n-k}_{k-1-j}y^{2n-2k+j}x^{k-1-j}$$
	By the equality (\ref{eq8}), we get 
	$$\phi^{2n-k-1}_x=a\binom{n}{k}(n-k)y^{2n-2k-2}x^k+\sum_{j=0}^{k-1} \frac{(2n-2k+1+j)a^{2n-k}_{k-1-j}}{2}y^{2n-2k-1+j}x^{k-1-j}\,.$$
	Integrating with respect to $x$, there exists a constant $a^{2n-k-1}_{0}\in F$ such that
	$$\phi^{2n-k-1}=a\binom{n}{k+1}y^{2n-2k-2}x^{k+1}+\sum_{j=0}^{k-1} a^{2n-k-1}_{k-j}y^{2n-2k-1+j}x^{k-j}+a^{2n-k-1}_{0}y^{2n-k-1}\,,$$
	where $\displaystyle a^{2n-k-1}_{k-j}= \frac{(2n-2k+1+j)a^{2n-k}_{k-1-j}}{2(k-j)}$ for all $0\leq j\leq k-1$. Thus, we obtain
	$$\phi^{2n-k-1}=a\binom{n}{k+1}y^{2n-2k-2}x^{k+1}+\sum_{j=0}^{k}a^{2n-k-1}_{k-j}y^{2n-2k-1+j}x^{k-j}\,,$$
	as wished, and this completes the recurrence.
	
	Now, replacing the expressions (\ref{eq7}) and (\ref{eq12}) in the equation (\ref{eqh}), we have  
	\begin{align*}
		\varphi
		&=\phi^{2n}+\sum_{k=1}^{n}\phi^{2n-k}z^k+\sum_{k=n+1}^{2n}\phi^{2n-k}z^k\nonumber\\
		&=ay^{2n}+ \sum_{k=1}^n \left(a\binom{n}{k}y^{2n-2k}x^k+\sum_{j=0}^{k-1} a^{2n-k}_{k-1-j}y^{2n-2k+1+j}x^{k-1-j}\right)z^k +\sum_{k=n+1}^{2n}\phi^{2n-k}z^k\\
		&=a\sum_{k=0}^{n}\binom{n}{k}y^{2n-2k}x^kz^k+ \sum_{k=1}^{n}\left(\sum_{j=0}^{k-1} a^{2n-k}_{k-1-j}y^{2n-2k+1+j}x^{k-1-j}\right)z^k+\sum_{k=n+1}^{2n}\phi^{2n-k}z^k\,.
	\end{align*}
	From  the last equality, we can write
	\begin{equation*}
		\varphi=a(xz+y^2)^n+z\psi, 
	\end{equation*}
	where $\psi\in F[x,y,z]$ is a homogeneous polynomial of degree equal to $2n-1$. By the linearity of the solutions of the equation (\ref{eq1}), we deduce that $z\psi$ is a solution, hence $\psi$ is also a solution.
	
	($b$) Let us suppose that: 
	\begin{equation}\label{eqhi1}
		\varphi=\phi^{2n+1}+\phi^{2n}z+\cdots+\phi^{1}z^{2n}+\phi^0z^{2n+1}\,,
	\end{equation}
	where $\phi^{2n+1-k}$ is a homogeneous  polynomial in $ F[x,y]$ of  degree equal to $2n+1-k$, for all $0\leq k\leq 2n+1$.  By Lemma \ref{psedp}, we get the following equalities:  
	\begin{align}	
		\phi^{2n+1}&=ay^{2n+1}\label{eqi7}\\	
		2y\phi^{2n-k}_x&=\phi^{2n+1-k}_y\,\,\, \textnormal{para todo}\,\,\, k=0,1,\dots, 2n-1\,.\label{eqi8}
	\end{align}
	As in the even case, we claim that there are constants $a,b_{k-j}^{2n-k}$,  with $0\leq k\leq n-1$ and $0\leq j\leq k$, such that:
	\begin{equation}\label{eqodd}
		\phi^{2n-k}=P^{2n+1}_{k}ay^{2n-2k-1}x^{k+1}+\sum_{j=0}^{k}b^{2n-k}_{k-j}y^{2n-2k+j}x^{k-j}\,,
	\end{equation}
	where 
	$
	P^{2n+1}_{k}=\frac{(2n+1)\cdot(2n+1-2)\cdots (2n+1-2k)}{2^{k+1}\cdot (k+1)!}
	$
	for all index $0\leq k\leq n-1$. 
	
	Replacing $k=0$ in the equation (\ref{eqi7}) and joining with  (\ref{eqi8}), we have
	\begin{equation*} 
		2y\phi^{2n}_x=\phi^{2n+1}_y=(2n+1)ay^{2n}\,,
	\end{equation*} 
	hence $\phi_x^{2n}=\cfrac{(2n+1)}{2}\, \, ay^{2n-1}$.
	Integrating with respect to $x$, we obtain
	\begin{equation*}
		\phi^{2n}=\cfrac{(2n+1)}{2}\,a y^{2n-1}x+b^{2n}_{0}y^{2n}
	\end{equation*}
	thus \eqref{eqodd} is true for $k=0$. 
	Suppose that we have \eqref{eqodd}. Taking the derivative with respect to $y$, we have
	$$\phi^{2n-k}_y=P^{2n+1}_{k}(2n-2k-1)ay^{2n-2k-2}x^{k+1}+\sum_{j=0}^{k}(2n-2k+j)b^{2n-k}_{k-j}y^{2n-2k+j-1}x^{k-j}\,.$$
	From the equality (\ref{eqi8}), we get 
	$$\phi^{2n-k-1}_x=\frac{P^{2n+1}_{k} (2n-2k-1)}{2}\, ay^{2n-2k-3}x^{k+1}+\sum_{j=0}^{k}\frac{(2n-2k+j)b^{2n-k}_{k-j}}{2}y^{2n-2k+j-2}x^{k-j}\,.$$
	Integrating with respect to $x$, we find a constant $b^{2n-k-1}_{0}\in F$ such that $\phi^{2n-k-1}$ is equal to
	$$P^{2n+1}_{k+1}ay^{2n-2k-3}x^{k+2}+\sum_{j=0}^{k}\frac{(2n-2k+j)b^{2n-k}_{k-j}}{2(k-j+1)}y^{2n-2k+j-2}x^{k-j+1}+b^{2n-k-1}_{0}y^{2n-k-1}\,,$$
	taking into account that $ \displaystyle P^{2n+1}_{k+1}=\frac{P^{2n+1}_{k}\cdot (2n-2k-1)}{2\cdot (k+2)}$. Defining $$\displaystyle b^{2n-k-1}_{k+1-j}=\frac{(2n-2k+j)b^{2n-k}_{k-j}}{2(k-j+1)}$$ for all $0\leq j\leq k $, we get
	$$\phi^{2n-k-1}=P^{2n+1}_{k+1}\, ay^{2n-2k-3}x^{k+2}+\sum_{j=0}^{k+1}b^{2n-k-1}_{k+1-j}y^{2n-2k+j-2}x^{k-j+1}$$
	as wished, and this completes the recurrence. In particular, replacing $k=n-1$ in \eqref{eqodd}, we obtain:
	\begin{equation*}
		\phi^{n+1}=\frac{1\cdot 3\cdots (2n+1)}{2^{n}\cdot n!}ayx^{n}+\sum_{j=0}^{n-1}b^{n+1}_{n-1-j}y^{j+2}x^{n-1-j}\,.
	\end{equation*}
	Considering the equation (\ref{eqi8}) for  $k=n$,  we get $2y\phi^{n}_x=\phi^{n+1}_y$, hence
	\begin{align}
		2y\phi^{n}_x&=\frac{1\cdot 3\cdots (2n+1)}{2^{n-1}\cdot n!}ax^{n}+\sum_{j=0}^{n-1}(j+2)b^{n+1}_{n-1-j}y^{j+1}x^{n-1-j}\,.
	\end{align}
	From this equality, we deduce that $a=0$, and consequently $\phi^{2n+1}=0$. Therefore,
	\begin{align*}
		\varphi=z\phi^{2n}+z^2\phi^{2n-1}+\cdots+ z^{2n}\phi^{1}+\phi^0 z^{2n+1}=z\psi\,,
	\end{align*}
	where $\psi\in F[x,y,z]$  is a homogeneous polynomial  of degree $2n$. Finally, arguing as in item ($a$), we deduce that $\psi$ is also a solution.
	
\end{proof}

\section{Proof of the main results}
\label{sectproofs}

We are now ready to prove the main theorems. 

\begin{proof}[Proof of Theorem \ref{Teorpr}]
	Suppose that $\varphi=p(xz+y^2,z)$, where $p$ is a polynomial in $F[t_1,t_2]$. Replacing the partial derivatives 
	\begin{align*}
		\varphi_x&= z\cdot p_{t_1}(xz+y^2,z)\,,\\  
		\varphi_y&= 2y\cdot p_{t_1}(xz+y^2,z)\,.
	\end{align*}
	in the linear PDE \eqref{eq1}, one readily verifies that $\varphi=p(xz+y^2,z)$ is a solution. Reciprocally, suppose that $\varphi\in F[x,y,z]$ is a solution of the equation  (\ref{eq1}). We shall prove that $\varphi$ has the form  $\varphi=p(xz+y^2,z)$, where $p$ is a polynomial in $F[t_1,t_2]$. By Lemma \ref{lemahom}, we can assume that $\varphi\in F[x,y,z]$ is a homogeneous polynomial. More precisely, we shall prove by recurrence the following: 
	\begin{equation}\label{eqhf}
		\varphi=\left\{ \begin{array}{rcl}
			\sum_{k=0}^{n}a_{2k}z^{2k}(xz+y^2)^{2n-2k}\,\,\,\,\,\,\,\,\,\,\,& \mbox{if} & \deg \varphi=2n\,,\\
			& & \\
			\sum_{k=1}^{n}a_{2k-1}z^{2k-1}(xz+y^2)^{2n-2k} & \mbox{if} & \deg \varphi=2n-1\,.
		\end{array}
		\right. 
	\end{equation}
	Indeed, applying Proposition \ref{pr1} for $n=1$, we obtain:
	\begin{equation*}
		\varphi=\left\{ \begin{array}{ccl}
			a(xz+y^2)+z\psi\,\,\,\,\,\,\,\,\,\,\,& \mbox{if} & \deg \varphi=2\,,\\
			& & \\
			z\psi_1  & \mbox{if} & \deg \varphi=1\,,
		\end{array}
		\right. 
	\end{equation*}
	where $\psi$ and $\psi_1$ are homogeneous polynomials with $\deg \psi=1$ and $\deg\psi_1=0$. Then there are constants $b,c\in F$ such that  $\psi=cz$ and $\psi_1=b\in F$, so \eqref{eqhf} is true for $n=1$.\\
	Let us suppose that \eqref{eqhf} holds for homogeneous polynomial solutions of degree $2n$ or $2n-1$.
	If $\deg\varphi=2n+2$, then by Proposition \ref{pr1} ($a$), we obtain that $\varphi=a(xz+y^2)^{2n+2}+z\psi$, where $\psi$ is a solution of the equation \eqref{eq1}. By the inductive hypothesis, we have 
	$$
	\psi=\sum_{k=1}^{n}a_{2k+1}z^{2k-1}(xz+y^2)^{2n-2k}\,.
	$$
	Hence,
	\begin{align*}
		\varphi&=a(xz+y^2)^{2n+2}+z\psi\\
		&=a(xz+y^2)^{2n+2}+z\left(\sum_{k=1}^{n}a_{2k+1}z^{2k-1}(xz+y^2)^{2n-2k}\right)
		\,.
	\end{align*}
	Therefore, $\varphi$ has the form $\varphi=p(xz+y^2,z)$, as wished. 
	On the other hand, if $\deg\varphi=2n+1$, then by Proposition \ref{pr1} ($b$), we obtain that $\varphi=z\psi$, where $\psi$ is a solution of the equation (\ref{eq1}) with $\deg \psi=2n$. By the inductive hypothesis, one has
	$$
	\psi=\sum_{k=0}^{n}a_{2k}z^{2k}(xz+y^2)^{2n-2k}\,.
	$$
	Hence,
	\begin{align*}
		\varphi&=z\psi\\
		&=z\left(\sum_{k=0}^{n}a_{2k}z^{2k}(xz+y^2)^{2n-2k}\right)\\
		&=\sum_{k=0}^{n}a_{2k}z^{2k+1}(xz+y^2)^{2n-2k}\\
		&=\sum_{k=1}^{n+1}a_{2k-2}z^{2k-1}(xz+y^2)^{2(n+1)-2k}\,,
	\end{align*}
	thus $\varphi$ has the form $\varphi=p(xz+y^2,z)$, and this completes the proof.
\end{proof}

\begin{proof}[Proof of the Corollary \ref{mainth}]
	One implication is straightforward. Suppose that the determinant of $ J(f,g,h)$ is a nonzero constant. By Proposition \ref{detedp}, $\varphi$ satisfies the linear PDE \eqref{eq1}. By Theorem \ref{Teorpr}, the solution of the equation \eqref{eq1} has the form $\varphi=p(xz+y^2,z)$, where $p$ is a polynomial in $F[t_1,t_2]$. Finally, by Proposition \ref{automp},  $(f,g,h)$ is an automorphism and its inverse is given in \eqref{pinverse}. 
\end{proof}

\section{Conclusions}

In Corollary \ref{mainth}, we show that the Nagata automorphisms of the form \eqref{eqnag} satisfy the well-known Jacobian conjecture. The proof is based on algebraic techniques for solving first-order linear partial differential equations. This result suggests a new strategy for approaching this longstanding open problem.


\begin{thebibliography}{00}
		\bibitem{Bass84} 	H. Bass.
	\newblock{\it A non-triangular action of $\mathbb{G}_a$ on $\mathbb{A}^3$},
	\newblock J. of Pure and Appl. Algebra 33, no.1, 1–5, (1984). 
	

	
	\bibitem{BiviaMZ} Bivià-Ausina, C.
	\newblock{\it Injectivity of real polynomial maps and \L ojasiewicz exponents at infinity},
	\newblock Math. Z. {\bf 257}, No. 4, 745--767, (2007).
	
	\bibitem{BH} Bivià-Ausina, C. and Huarcaya, J.\,A.\,C.
	\newblock{\it Growth at infinity and index of polynomial maps},
	\newblock J. Math. Anal. Appl. {\bf 422}, 1414--1433,  (2015).
	
	\bibitem{CKT} Cygan, E., Krasi\'nski, T. and Tworzewski, P.
	\newblock{\it Separation of algebraic sets and the \L ojasiewicz exponent of polynomial mappings},
	\newblock Invent. Math. {\bf 136}, no. 1, 75--87,  (1999). 
	
			\bibitem{Freu17} 	Freudenburg, G. 
	\newblock{\it  Algebraic Theory of Locally Nilpotent Derivations},
	\newblock  (2ª ed.) Springer, 
	(2017).
	

\bibitem{Kur14} 	Kuroda, S.
\newblock{\it  The Nagata Type Polynomial Automorphisms and Rectifiable Space Lines},
\newblock Communications in Algebra, 42(10), 4451–4455. 
 (2014).  

	\bibitem{Nag72} Nagata, M.
On the automorphism group of $k[x, y]$. 	\textit{ Lectures in Math., Kyoto Univ.}, Kinokuniya, Tokyo.  (1972). 

	\bibitem{Kra07} 
	T. Krasinski, On the Łojasiewicz exponent at infinity of polynomial mappings, 	\textit{ Acta
Math. Vietnam.} 32, No. 2–3, 189–203,  (2007).

\bibitem{Len} Lenarcik, A.
\newblock{\it On the \L ojasiewicz exponent of the gradient of a polynomial function},
\newblock Ann. Polon. Math. {\bf 71}, No. 3, 211–239,  (1999).

	\bibitem{Plos10}
Płoski A. Semicontinuity of the Łojasiewicz exponent. Acta Math.  48 :103-110, (2010). 

	\bibitem{Plos85}
A. P\l{}oski, On the growth of proper polynomial mappings, \textit{Ann. Polon. Math.} 45,
297–309,  (1985).
	
	\bibitem{SU04} Shestakov, I. P.,  Umirbaev, U. U. 
The Tame and the Wild Automorphisms of Polynomial Rings in Three Variables.	\textit{J. Amer. Math. Soc.}, 17(1): 197–227. (2004). 
	
	\bibitem{SU05} Shestakov, I. P.,  Umirbaev, U. U.  Poisson brackets and two-generated subalgebras of rings of polynomials. \textit{J. Amer. Math. Soc.} 17(1): 181-196.  (2004). 
\end{thebibliography}
\end{document}